\documentclass[11pt]{amsart}
\usepackage{relsize}
\usepackage{amstext}
\usepackage{amsfonts}
\usepackage{amssymb}
\usepackage{amsbsy}
\usepackage{latexsym}
\usepackage{xy}
\usepackage{tikz-cd}
\usepackage{hhline}
\xyoption{all}

\newcounter{num}[section] %

\newenvironment{theo}
{\refstepcounter{num}%
\bigskip\noindent{\bf Theorem~\arabic{section}.\arabic{num}. }\it}

\newenvironment{cor}
{\refstepcounter{num}%
\bigskip\noindent{\bf Corollary~\arabic{section}.\arabic{num}. }\it}

\newenvironment{lemma}
{\refstepcounter{num}%
\bigskip\noindent{\bf Lemma~\arabic{section}.\arabic{num}. }\it}

\newenvironment{eq}{\begin{equation}}{\end{equation}}

\renewcommand{\Ref}[1]{(\ref{#1})}

\newcommand{\si}{\sigma}
\newcommand{\al}{\alpha}
\newcommand{\be}{\beta}
\newcommand{\ga}{\gamma}
\newcommand{\la}{\lambda}
\newcommand{\de}{\delta}

\newcommand{\LA}{\langle}
\newcommand{\RA}{\rangle}
\newcommand{\ov}[1]{\overline{#1}}
\newcommand{\un}[1]{{\underline{#1}} }

\newcommand{\lin}{\mathop{\rm lin}}
\newcommand{\tr}{\mathop{\rm tr}}

\newcommand{\mdeg}{\mathop{\rm mdeg}}

\newcommand{\Char}{\mathop{\rm char}}

\newcommand{\algA}{\mathcal{A}}    

\newcommand{\FF}{{\mathbb{F}}}   
\newcommand{\NN}{{\mathbb{N}}}

\newcommand{\Q}{\mathcal{Q}}    

\newcommand{\Y}{\LA Y\RA}
\renewcommand{\L}{\LA L\RA}

\newcommand{\EL}{\LA \widetilde{L}\RA}

\newcommand{\PhiLin}[1]{\Xi_{#1}}  

\newcommand{\PhiELin}[1]{\widetilde{\Xi}_{#1}}  

\begin{document}
\title[Indecomposable orthogonal invariants]{Indecomposable orthogonal invariants of several matrices over a field of positive characteristic}

\thanks{The results from Section 3 were supported by RFBR 16-31-60111 (mol\_a\_dk) and the results from Section 4 were supported by CNPq 307712/2014-1}

\author{Artem Lopatin}
\address{Artem Lopatin\\ 
State University of Campinas, 651 Sergio Buarque de Holanda, 13083-859 Campinas, SP, Brazil}
\email{dr.artem.lopatin@gmail.com (Artem Lopatin)}
 

\begin{abstract} 
We consider the algebra of invariants of $d$-tuples of $n\times n$
matrices under the action of the orthogonal group by simultaneous conjugation over an infinite field of  characteristic $p$ different from two.  It is well-known that this algebra is generated by the coefficients of the characteristic polynomial of all products of generic and transpose generic $n\times n$ matrices. We establish that in case $0<p\leq n$ the maximal degree of indecomposable invariants tends to infinity as $d$ tends to infinity. In other words, there does not exist a constant $C(n)$ such that it only depends on $n$  and the considered algebra of invariants is generated by elements of degree less than $C(n)$ for any $d$.  This result is well-known in case of the action of the general linear group. On the other hand, for the rest of $p$ the given phenomenon  does not hold. We investigate the same problem for the cases of symmetric and skew-symmetric matrices.  


\end{abstract}

\maketitle

\section{Introduction}\label{section_intro}
In this paper all vector spaces and algebras are over an infinite field $\FF$ of characteristic $p=\Char{\FF}\geq 0$.  By an algebra we always mean an associative algebra. Whenever we consider the orthogonal group $O(n)$, that is, the group of all $n \times n$ matrices $A$ over $\FF$ such that $AA^T=E$, we assume that $p\neq2$.

Assume that $n>1$ and $d\geq 2$. To define the algebra of {\it matrix $O(n)$-invariants} $R^{O(n)}$ we consider the polynomial algebra 
$$R=R_{n,d}=\FF[x_{ij}(k)\,|\,1\leq i,j\leq n,\, 1\leq k\leq d],$$ 
and denote by 
$$X_k=\left(\begin{array}{ccc}
x_{11}(k) & \cdots & x_{1n}(k)\\
\vdots & & \vdots \\
x_{n1}(k) & \cdots & x_{nn}(k)\\
\end{array}
\right)$$
an $n\times n$ {\it generic} matrix, where  $1\leq k\leq d$. Given an $n\times n$ matrix $A$, we write $\si_t(A)$ for the $t^{\rm th}$ coefficient of the characteristic polynomial of $A$. Then $R^{O(n)}$ is the subalgebra of $R$
generated by $\si_t(A)$, where $1\leq t\leq n$ and $A$ ranges over all monomials in the matrices $X_1,\ldots,X_d$, $X_1^T,\ldots,X_d^T$.  Considering only products of generic matrices we obtain the algebra $R^{GL(n)}$ of {\it matrix $GL(n)$-invariants}.  Denote by $X_k^{+}$ and $X_k^{-}$ the symmetric and skew-symmetric generic matrices, respectively. Namely, 
$$
(X_k^{+})_{ij}=
\left\{ 
\begin{array}{rc}
x_{ij}(k),& \text{ if } i\geq j\\
x_{ji}(k),& \text{ if } i< j\\
\end{array}
\right.\qquad\text{and}\qquad
(X_k^{-})_{ij}=
\left\{ 
\begin{array}{rc}
x_{ij}(k),& \text{ if } i> j\\
0 ,& \text{ if } i= j\\
-x_{ji}(k),& \text{ if } i< j\\
\end{array}
\right.,
$$
where $(B)_{ij}$ stands for the $(i,j)^{\rm th}$ entry of a matrix $B$. The algebra $R_{+}^{O(n)}$ ($R_{-}^{O(n)}$, respectively) of {\it $O(n)$-invariants of symmetric} ({\it skew-symmetric}, respectively) {\it matrices} is generated  by $\si_t(A)$, where $A$ ranges over monomials in $X_1^{+},\ldots,X_d^{+}$  ($X_1^{-},\ldots,X_d^{-}$, respectively). Note that in each of the considered  cases we can assume that $A$ is {\it primitive}, i.e., is not equal to a power of a shorter monomial. By the Hilbert--Nagata Theorem on invariants, these algebras of invariants are finitely generated, but the mentioned generating systems are infinite. 

The above given generators of $R^{GL(n)}$, $R^{O(n)}$ were described  in~\cite{Sibirskii_1968}, \cite{Procesi76}, \cite{Donkin92a}, \cite{Zubkov99} and the generators of $R_{+}^{O(n)}$ and $R_{-}^{O(n)}$ were given in~\cite{Lopatin_so_inv} (see also~\cite{ZubkovI}). In case $p=0$ relations  between generators were independently computed in~\cite{Razmyslov74} and~\cite{Procesi76}.  For an arbitrary $p$ relations for $GL(n)$ and $O(n)$-invariants were established in~\cite{Zubkov96},~\cite{Lopatin_Orel},~\cite{Lopatin_Ofree}. 

For $f\in R$ denote by $\deg{f}$ its {\it degree} and by $\mdeg{f}$ its {\it multidegree}, i.e.,
$\mdeg{f}=(t_1,\ldots,t_d)$, where $t_k$ is the total degree of the polynomial $f$ in $x_{ij}(k)$, $1\leq i,j\leq n$, and $\deg{f}=t_1+\cdots+t_d$. Since $\deg{\si_t(Y_1\cdots Y_s)}=ts$, where $Y_k$ is a generic or a transpose generic matrix, the algebra $R^{O(n)}$ has $\NN$-grading by degrees and $\NN^d$-grading by multidegrees, where $\NN$ stands for non-negative integers. 

Consider an $\NN$-graded unital algebra $\algA$ with the component of degree zero equal to $\FF$. Denote by $\algA^{+}$ the subalgebra generated by homogeneous elements of positive degree. A set $\{a_i\} \subseteq \algA$ is a minimal (by inclusion) homogeneous system of generators
(m.h.s.g.) of $\algA$ as a unital algebra~if and only if the $a_i$'s are $\NN$-homogeneous and $\{\ov{a_i}\}\cup\{1\}$ is a basis of the vector space $\ov{\algA}={\algA}/{(\algA^{+})^2}$. If we consider $a\in\algA$ as an element of $\ov{\algA}$, then we usually omit the bar and write $a\in\ov{\algA}$ instead of $\ov{a}$. An element $a\in \algA$ is called {\it
decomposable} if $a=0$ in $\ov{\algA}$. In other words, a decomposable
element is equal to a polynomial in elements of strictly lower degree. Therefore  the highest degree $D(\algA)$ of indecomposable elements of  $\algA$  is equal to the least upper bound for the degrees of elements of a m.h.s.g.~for $\algA$.

It is well-known that for $0<p\leq n$ we have $D(R^{GL(n)})\to\infty$ as $d\to\infty$ (see~\cite{DKZ_2002}). On the other hand, in case $p=0$ or $p>n$ there exists an upper bound on $D(R^{GL(n)})$, which does not depend on $d$. Namely, it follows from the well-known Nagata--Higman Theorem (see~\cite{Nagata_1953} and~\cite{Higman_1956}), which at first was proved by Dubnov and Ivanov~\cite{Dubnov_1943} in 1943, that  $D(R^{GL(n)})\leq 2^{n}$. Since 
$$D(R_{+}^{O(n)})\leq D(R^{GL(n)}),\; D(R_{-}^{O(n)})\leq D(R^{GL(n)}),\; \text{ and } D(R_{n,d}^{O(n)})\leq D(R_{n,2d}^{GL(n)}),$$
we can see that there also exists an upper bound on $D(\algA)$ which does not depend on $d$, where $\algA$ is $R^{O(n)}$, $R_{+}^{O(n)}$ or $R_{-}^{O(n)}$. In this paper we consider the case of $0<p\leq n$ for the orthogonal invariants as well as for the orthogonal symmetric and skew-symmetric invariants: 

\begin{theo}\label{theo1}
\begin{enumerate}
\item[(a)] Assume that $2<p\leq n$. Then the invariant $\tr(X_1\cdots X_d)$ is indecomposable in $R^{O(n)}$ and $\tr(X_1^{+}\cdots X_d^{+})$ is indecomposable in  $R_{+}^{O(n)}$. 

\item[(b)] Assume that  $2<p\leq \frac{n-1}{2}$ and $d$ is even. Then the invariant $\tr(X_1^{-}\cdots X_d^{-})$ is indecomposable in $R_{-}^{O(n)}$. 
\end{enumerate}
\end{theo}
\medskip

Part~(a) of Theorem~\ref{theo1} is proven in Section~\ref{section3} and part~(b) is proven in Section~\ref{section4}.  Theorem~\ref{theo1} immediately implies the following corollary:

\begin{cor}\label{cor1}
\begin{enumerate}
\item[(a)]  $D(R^{O(n)})\geq d$ and $D(R_{+}^{O(n)})\geq d$ in case $2<p\leq n$;

\item[(b)]  $D(R_{-}^{O(n)})\geq d-1$ in case  $2<p\leq \frac{n-1}{2}$.
\end{enumerate}
\end{cor}
\medskip


M.h.s.g.-s for matrix $GL(2)$ and $O(2)$-invariants are known (see~\cite{Sibirskii_1968},~\cite{Procesi_1984},~\cite{DKZ_2002}). For example, a m.h.s.g.~for $R^{GL(2)}$ is the following set:
$$\tr(X_{i_1}\cdots X_{i_k}), \det(X_i), \text{ where } 1\leq i_1<\cdots<i_k\leq  d, \; 1\leq i\leq d,$$
where $k=1,2,3$ in case $p\neq 2$ and $k>0$ in case $p=2$. A m.h.s.g.~for matrix $GL(3)$-invariants was established in~\cite{Lopatin_Comm1} and~\cite{Lopatin_Comm2}. In case $p=0$ and $d=2$ a m.h.s.g.~for $R^{GL(n)}$ was established in~\cite{Drensky_Sadikova_4x4} for $4\times 4$ matrices (see also~\cite{Djokovic_GLinv_2007}) and in~\cite{Djokovic_GLinv_2007} for $5\times 5$ matrices.  

In~\cite{Lopatin_O3} the following estimations were given:
\begin{enumerate} 
\item[$\bullet$] If $p=3$, then $2d+4\leq D(R^{O(3)})\leq 2d+7$.

\item[$\bullet$] If $p \neq 2,3$, then $D(R^{O(3)})=6$.
\end{enumerate}
A m.h.s.g.~for $R^{O(3)}_{-}$ in the case of arbitrary $d$ was given in~\cite{Lopatin_Oskew}. In particular, it was proven that $D(R^{O(3)}_{-})=3$ for $d>2$ and $p\neq2$. Thus the statement of part~(b) of Corollary~\ref{cor1} does not hold in general when $\frac{n-1}{2}<p\leq n$. We formulate the following conjecture:

\begin{conj}\label{conj}
In case $\frac{n-1}{2}<p\leq n$ we have $D(R_{-}^{O(n)})\leq C(n)$ for some  $C(n)$ that does not depend on $d$.
\end{conj}
\medskip

A  m.h.s.g.~for $R^{O(4)}_{-}$ in case  $d=2$ was recently established in~\cite{Lopatin_mgs_skewO4}. In case $p=0$ m.h.s.g.-s~for $O(n)$- and $SO(n)$-invariants of one matrix were given in~\cite{Djokovic_2005} and~\cite{Djokovic_Osmall} in case $n\leq 5$ and m.h.s.g.-s~for symplectic invariants  were established in~\cite{Djokovic_Sp} in case $(n,d)$ is $(4,2)$ or $(6,1)$.

\section{Notation and auxiliary results}\label{section_notations}

For a vector $\un{k}=(k_1,\ldots,k_t)$ we denote $\#\un{k}=t$. In case $\un{k}\in\NN^t$ we set $|\un{k}|=k_1+\cdots +k_t$. Given a symbol $\de\in\{1,T\}$ and a matrix $A$, we write $A^{\de}$ for $A$ ($A^T$, respectively) in case $\de=1$ ($\de=T$, respectively). Similar notation we apply in case when instead of a matrix $A$ we consider an element $a$ from some semigroup with involution $T$ (as an example, see the definition of $\mu$ in the beginning of Section~\ref{section4}). In this paper we use the following notions: 
\begin{enumerate}
\item[$\bullet$] The semigroup $\Y$ (without unity) is freely generated by {\it letters}  $x_1,\ldots,x_d,x_1^T,\ldots,x_d^T$ and the monoid $\Y^{\#}=\Y\sqcup\{1\}$.

\item[$\bullet$] The involution ${}^T:\Y\to\Y$ is defined by $(x_i)^{T}=x_i^T$ and $(x_i^T)^T=x_i$ for all $i$ and $(a_1\cdots a_r)^T=a_r^T\cdots a_1^T$ for any  $a_1,\ldots,a_r\in\Y$.

\item[$\bullet$] We say that $a,b\in\Y$ are {\it cyclicly equivalent} and write $a\stackrel{c}{\sim} b$
if $a=a_1a_2$ and $b=a_2a_1$ for some $a_1,a_2\in\Y^{\#}$. If $a\stackrel{c}{\sim} b$ or $a\stackrel{c}{\sim} b^T$, then we say that $a$ and $b$ are {\it equivalent} and write $a\sim b$. 

\item[$\bullet$] $\L$ is the set of multilinear elements of $\Y$, i.e., of monomials $x_{\si(1)}^{\de_1}\cdots x_{\si(d)}^{\de_d}$, where $\si\in S_{d}$ is a permutation and $\de_i\in\{1,T\}$ for all $i$.  Let $\EL\subset\L$ be the subset of all elements that start with the letter $x_1$. In other words, $\EL$ is a set of distinct representatives of the ${\sim}$-equivalence classes of elements of $\L$.

\item[$\bullet$] $\tr\L$ is the $\FF$-span of ``symbolic''{} elements $\tr(a)$ for $a$ from $\Y$  and $\tr\EL$ is the $\FF$-span of ``symbolic''{}  elements $\tr(a)$ for $a$ from $\EL$. 

\item[$\bullet$] $\lin(\ov{R^{O(n)}})$ the $\FF$-span of $\tr(X_{\si(1)}^{\de_1}\cdots X_{\si(d)}^{\de_d})$ in $\ov{R^{O(n)}}$, where $\si\in S_d$ and $\de_i\in\{1,T\}$ for all $i$. 

\item[$\bullet$] $\PhiLin{n,d}:\tr\L\to \lin(\ov{R^{O(n)}})$ is the linear surjective map defined by $$\tr(x_{\si(1)}^{\de_1}\cdots x_{\si(d)}^{\de_d})\to 
\tr(X_{\si(1)}^{\de_1}\cdots X_{\si(d)}^{\de_d}\!),$$
where $\si, \de_i$ are the same as above. Similarly we define $\PhiELin{n,d}:\tr\EL\to \lin(\ov{R^{O(n)}})$.

\item[$\bullet$] $I_{n,d}$ is the kernel of $\PhiLin{n,d}$, i.e., the space of relations between indecomposable multilinear elements of $R^{O(n)}$, and $\widetilde{I}_{n,d}$ is the kernel of $\PhiELin{n,d}$.
\end{enumerate} 

A vector $(\un{u};\un{v};\un{w})$ with $\un{u}=(u_1,\ldots,u_t)$, $\un{v}=(v_1,\ldots,v_r)$, $\un{w}=(w_1,\ldots,w_r)$, $u_1,\ldots,w_r\in\Y$ is called a {\it multilinear triple} of $\L$ if $u_1\cdots w_r\in\L$.  

In order to define the element $\si_{\rm lin}(\un{u};\un{v};\un{w})$ of $\tr\L$ for a multilinear triple $(\un{u};\un{v};\un{w})$ of $\L$ we consider the quiver (i.e., the oriented graph) $\Q=\Q_{t,r}$:
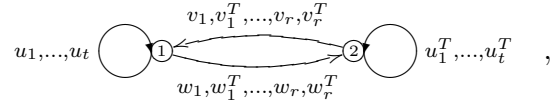
\begin{center}
\begin{tikzcd}[cells={nodes={draw=black, circle}}, column sep=10em]
1
\arrow[out=-130, in=130, loop, distance=2.5cm,"\mathlarger{\al_1,\ldots,\al_t}"]
\arrow[bend left]{r}{\mathlarger{\be_1,\be_1^T,\ldots,\be_r,\be_r^T}}
& 2 
\arrow[out=50, in=-50, loop, distance=2.5cm,"\mathlarger{\al_1^T,\ldots,\al_t^T}"]
\arrow[bend left]{l}{\mathlarger{\ga_1,\ga_1^T,\ldots,\ga_r,\ga_r^T}} \\
\end{tikzcd}
\end{center}
%
where there are $2r$ arrows from vertex $2$ to vertex $1$ as well as from $1$ to $2$, and there are $t$ loops at each of the two vertices.  For an arrow $a$ denote by $a'$ its head and by $a''$ its tail (i.e., $\be'_i=\ga''_i=1$ and so on). A sequence of arrows $a=a_1\cdots a_s$ of $\Q$ is a {\it path} of $\Q$ if $a_i''=a_{i+1}'$ for all $1\leq i< s$. The head of the path $a$ is $a'=a_1'$ and the tail is $a''=a_s''$. Note that $a^T=a_s^T\cdots a_1^T$ is also a path in $\Q$, where by definition $(\al_i)^T=\al_i^T$, $(\be_j)^T=\be_j^T$, $(\ga_j)^T=\ga_j^T$ and $(\al_i^T)^T=\al_i$, $(\be_j^T)^T=\be_j$, $(\ga_j^T)^T=\ga_j$ for all $i,j$. A path $a$ is {\it closed} if $a'=a''$.
For closed paths $a,b$  we write $a\stackrel{c}{\sim} b$ if $a=a_1a_2$ and $b=a_2a_1$ for some paths $a_1,a_2$, where a path $a_2$ can be empty. If $a\stackrel{c}{\sim} b$ or  $a\stackrel{c}{\sim} b^T$, then we say that $a$ and $b$ are equivalent and write $a\sim b$.

Given a path $a$ in $\Q$ and an arrow $b$ of $\Q$, denote by $\deg_b(a)$ the number of letters $b$ in $a$. Introduce the multidegree of a path $a$ in $\Q$ as follows:
$$\mdeg(a)=(\un{k};\un{l};\un{m})$$
for $\un{k}=(k_1,\ldots,k_t)$, $\un{l}=(l_1,\ldots,l_r)$, $\un{m}=(m_1,\ldots,m_r)$ with 
$$k_i=\deg_{\al_i}(a)+\deg_{\al_i^T}(a),\;\;
l_j=\deg_{\be_j}(a)+\deg_{\be_j^T}(a),\;\; 
m_j=\deg_{\ga_j}(a)+\deg_{\ga_j^T}(a)$$
for all $i,j$. Note that the path $a$ is closed if and only if $\sum_j l_j=\sum_j m_j$. 

A path $a$ is called {\it multilinear} if $\mdeg(a)=(1^t;1^r;1^r)$, where $1^t=(1,\ldots,1)$ ($t$ times); in particular, $a$ is closed.
Denote by $\Omega=\Omega_{t,r}$ the set of all multilinear paths $a=a_1 \cdots a_{t+2r}$ in $\Q$ (where $a_1,\ldots,a_{t+2r}$ are arrows in $\Q$) with $$a_1=
\left\{
\begin{array}{cc}
\al_1,& \text{if } r=0 \\
\be_1,& \text{if } r>0  \\
\end{array}
\right..$$ 

For a path $a$ we set
$$\xi_a=t+\sum_j\deg_{\be_j}(a) +\sum_j\deg_{\ga_j}(a).$$  

Finally, for a multilinear triple  $(\un{u};\un{v};\un{w})$  of $\L$ we define the next element of $\tr\L$:  
$$\si_{\rm lin}(\un{u};\un{v};\un{w})=\sum\limits_{a\in\Omega} (-1)^{\xi_a} \tr(a(\un{u};\un{v};\un{w})),$$
where $a(\un{u};\un{v};\un{w})$ stands for the result of substitutions $\al_i^{\de}\to u_i^{\de}$, $\be_j^{\de}\to v_j^{\de}$, $\ga_j^{\de} \to w_j^{\de}$  ($1\leq i\leq t$, $1\leq j\leq r$, $\de\in\{1,T\}$) in $a$. Note that here $a(\un{u};\un{v};\un{w})\in \L$. 

Consider a linear map $\pi:\tr\L \to \tr\EL$ defined by $\tr(a)\to \tr(b)$, where  $b\in\EL$ satisfies $b\sim a$. Applying $\pi$, we can consider any element of $\tr\L$ as an element of $\tr\EL$.  As an example, see the formulation of Lemma~\ref{L1}. Obviously, the following diagram is commutative.
$$ 
\begin{picture}(0,80)
\put(0,85){%
\put(-75,-33){$\tr\L$}%
\put(50,-33){$\tr\EL$}%
\put(-44,-30){\vector(1,0){87}}%
\put(50,-40){\vector(-3,-2){35}}%
\put(-50,-40){\vector(3,-2){35}}%
\put(-23,-78){$\lin(\ov{R^{O(n)}})$}%
\put(-3,-25){$\scriptstyle\pi$}%
\put(-35,-48){$\scriptstyle\PhiLin{n,d}$}%
\put(17,-48){$\scriptstyle\PhiELin{n,d}$}%
}%
\end{picture}
$$%

\begin{lemma}\label{L1}
\begin{enumerate}
\item[(a)] Given $t,r\geq0$, $\Omega_{t,r}$ is a set of representatives of all $\sim$-equivalence classes of multilinear paths in $\Q_{t,r}$.

\item[(b)] Given a multilinear triple $(\un{u};\un{v};\un{w})$ of $\L$, the following equality holds in $\tr\EL$: 
$$\si_{\rm lin}(\un{u};\un{v};\un{w})=\sum\limits (-1)^{\xi_a} \tr(a(\un{u};\un{v};\un{w})),$$
where the sum ranges over all $\sim$-equivalence classes in $\Q$.
\end{enumerate}
\end{lemma}
\begin{proof}Since elements of $\Omega_{t,r}$ are multilinear, part~(a) is obvious. Let us remark that for closed paths $a,b$ in $\Q_{t,r}$ with $a\sim b$ we have
\begin{eq}\label{eq1}
(-1)^{\xi_{a}}=(-1)^{\xi_{b}}.
\end{eq}%
\noindent{}To prove~\Ref{eq1} it is enough to notice that it holds in case $b=a^T$. Hence, part~(a) and equality~\Ref{eq1} imply part~(b). 
\end{proof}

\begin{lemma}\label{L2}
\begin{enumerate}
\item[(a)] The space $I_{n,d}\subset \tr\L$ is the $\FF$-span of 
\begin{enumerate}
\item[(A)] $\tr(ab)-\tr(ba)$ for $ab$ from $\L$;

\item[(B)] $\tr(a)-\tr(a^T)$ for $a$ from $\L$;

\item[(C)] $\si_{\rm lin}(\un{u};\un{v};\un{w})$ for a multilinear triple $(\un{u};\un{v};\un{w})$ of $\L$, if $t+2r=n+1$ for $t=\#\un{u}$ and $r=\#\un{v}=\#\un{w}$.
\end{enumerate}

\item[(b)] The space $\widetilde{I}_{n,d}\subset \tr\EL$ is the $\FF$-span of images of (C) with respect to $\pi$. 
\end{enumerate}
\end{lemma}
\begin{proof}
To prove part~(b) we apply Theorem 6.3 and Remarks 6.4, 6.5 from~\cite{Lopatin_JPAA2013} together with part (b) of Lemma~\ref{L1}.

Note that the kernel of $\pi$ is generated by elements from (A) and (B). This fact together with part~(b) imply part~(a).
\end{proof}

Given a sequence $\un{a}=(a_1,\ldots,a_k)$, where $a_1,\ldots,a_k\in\Y$, and $\si\in S_k$, we write  $\un{a}^T$ for $(a_1^T,\ldots,a_k^T)$ and $\un{a}_{\si}$ for $(a_{\si(1)},\ldots,a_{\si(k)})$.

\begin{lemma}\label{L3}
For a multilinear triple $(\un{u};\un{v};\un{w})$ of $\L$ the following elements of $\tr\L$ belong to the $\FF$-span of elements of types (A) and (B) from  Lemma~\ref{L2}:
\begin{enumerate}
\item[(a)] $\si_{\rm lin}(\un{u};\un{v};\un{w}) - \si_{\rm lin}(\un{u}^T;\un{w};\un{v})$;  

\item[(b)] $\si_{\rm lin}(\un{u};\un{v};\un{w}) - \si_{\rm lin}(\un{u};\un{v}^T;\un{w}^T)$.  
\end{enumerate}
\end{lemma}
\begin{proof}
Given a multilinear path $a$ in $\Q_{t,r}$, denote by $b$ the result of substitutions $\be_j\to \be_j^T$, $\ga_j\to \ga_j^T$ ($1\leq j\leq r$) in $a$. Then  $(-1)^{\xi_{a}}=(-1)^{\xi_{b}}$, $b$ is the multilinear path in $\Q_{t,r}$ and
\begin{eq}\label{eqNew1}
\tr(a(\un{u};\un{v};\un{w})) = \tr(b(\un{u}^T;\un{w};\un{v})) = \tr(b(\un{u};\un{v}^T;\un{w}^T)).
\end{eq}%
As an example, for $t=2$, $r=1$ and $a=\al_1 \be_1 \al_2^T \ga_1$ we have $b=\al_1 \be_1^T \al_2^T \ga_1^T$ and equalities~\Ref{eqNew1} are equivalent to $\tr(u_1 v_1 u_2^T w_1) = \tr(u_1^T w_1^T u_2 v_1^T) = \tr(u_1 v_1 u_2^T w_1)$. Lemma~\ref{L1} together with equalities~\Ref{eqNew1} conclude the proof of the lemma. \end{proof}

\begin{lemma}\label{Lnew1}
For a multilinear triple $(\un{u};\un{v};\un{w})$ of $\L$ and permutations $\si\in S_t$, $\nu, \eta\in S_r$ the following element of $\tr\L$ belongs to the $\FF$-span of elements of types (A) and (B) from  Lemma~\ref{L2}:
$$\si_{\rm lin}(\un{u};\un{v};\un{w}) - \si_{\rm lin}(\un{u}_{\si};\un{v}_{\nu};\un{w}_{\eta}).$$ 
\end{lemma}
\begin{proof} 
For any $a\in\Omega_{t,r}$ we have $a(\un{u}_{\si};\un{v}_{\nu};\un{w}_{\eta}) = a_{\si,\nu,\eta}(\un{u};\un{v};\un{w})$, where $a_{\si,\nu,\eta}$ is the result of substitutions  $\al_i^{\de}\to \al_{\si(i)}^{\de}$, $\be_j^{\de}\to \be_{\nu(j)}^{\de}$, $\ga_j^{\de} \to \ga_{\eta(j)}^{\de}$  ($1\leq i\leq t$, $1\leq j\leq r$, $\de\in\{1,T\}$) in $a$. By the definition of $\si_{\rm lin}$, 
$$\si_{\rm lin}(\un{u}_{\si};\un{v}_{\nu};\un{w}_{\eta}) = \sum_{a\in \Omega_{t,r}} (-1)^{\xi_a}\tr(a_{\si,\nu,\eta}(\un{u};\un{v};\un{w})).$$
Obviously, $a_{\si,\nu,\eta}$ is a multilinear path in $\Q_{t,r}$ and $\xi_a = \xi_{a_{\si,\nu,\eta}}$. Note that there exists $c_a\in\Omega_{t,r}$ with
$a_{\si,\nu,\eta} \sim c_a$. In particular, 
$$\tr(a_{\si,\nu,\eta}(\un{u};\un{v};\un{w}))-\tr(c_a(\un{u};\un{v};\un{w}))$$
belongs to the $\FF$-span of elements of types (A) and (B) from  Lemma~\ref{L2}. 
Moreover, for any $a, b\in \Omega_{t,r}$ we have
$$a\neq b \; \Longleftrightarrow\; a\not\sim b \; \Longleftrightarrow\; a_{\si,\nu,\eta}\not\sim b_{\si,\nu,\eta}.$$
Then $c_a$ ranges over $\Omega_{t,r}$ whenever  $a$ ranges over $\Omega_{t,r}$. The lemma is proven.
\end{proof}

\section{Multilinear matrix $O(n)$-invariants}\label{section3}

In this section we prove part~(a) of Theorem~\ref{theo1}.

\begin{lemma}\label{lemma1}
Assume that $0<p\leq n$. Then the sum of coefficients of an element $\sum_i \al_i \tr(a_i)$ of $I_{n,d}$ is zero, i.e., $\sum_i \al_i=0$.
\end{lemma}
\begin{proof}
It is enough to verify the statement of the lemma for elements of the types (A), (B) and (C) from Lemma~\ref{L2}. Clearly, the lemma holds for elements of the types (A) and (B). 

Consider a multilinear triple $(\un{u};\un{v};\un{w})$ of $\L$ with $|\un{u}|=t$, $|\un{v}|=r$ and $t+2r=n+1$.  Then the sum of coefficients of $\si_{\rm lin}(\un{u};\un{v};\un{w})$ is 
$$\la=\sum_{a\in\Omega} (-1)^{\xi_a}.$$

Assume that $r\geq1$. Denote by $\Omega_0$ the set of all $a$ from $\Omega$ with $\deg_{\ga_1}(a)=1$. Consider a map $\psi:\Omega_0\to\Omega$ that substitutes the arrow $\ga_1$ in $a\in\Omega_0$ by $\ga_1^T$. By the definition, $\Omega$ is the  disjoint union of $\Omega_0$ and $\psi(\Omega_0)$. Moreover, 
$$(-1)^{\xi_a}+(-1)^{\xi_{\psi(a)}}=0.$$ 
Thus $\la=0$.

Assume that $r=0$. Then $t=n+1$ and $\Omega$ consists of all monomials $\al_1 \al_{\si(2)} \cdots \al_{\si(t)}$ for a permutation $\si\in S_t$ with $\si(1)=1$. Hence $\la = (-1)^t (t-1)! =0$, since $0<p\leq n$.
\end{proof}

\begin{proof} Now we can prove part (a) of Theorem~\ref{theo1}. Since the sum of coefficients of the element $\tr(x_1\cdots x_d)$ of $\LA L\RA$ is one, then Lemma~\ref{lemma1} implies that the invariant $\tr(X_1\cdots X_d)$ is indecomposable in $R^{O(n)}$. 

Assume that $\tr(X_1^{+}\cdots X_d^{+})$ is decomposable in $R_{+}^{O(n)}$. Then $F=\tr((X_1+X_1^T)\cdots (X_d+X_d^T))$ is decomposable in $R^{O(n)}$.  Note that
$$F=\sum \tr(X_1^{\de_1}\cdots X_d^{\de_d}),$$
where the sum ranges over all $\de_1,\ldots,\de_d$ from $\{1,T\}$. Hence the element 
$$f=\sum \tr(x_1^{\de_1}\cdots x_d^{\de_d}) \in \tr\L$$
belongs to $I_{n,d}$, where the sum ranges over all $\de_1,\ldots,\de_d$ from $\{1,T\}$. The sum of coefficients of $f$ is $2^d$; a contradiction to Lemma~\ref{lemma1}.
\end{proof}

\section{Multilinear skew-symmetric matrix $O(n)$-invariants}\label{section4}

In this section we prove part~(b) of Theorem~\ref{theo1}. Consider the linear map $\mu: \tr\L\to \FF$ such that for $\de_1,\ldots,\de_d$ from $\{1,T\}$ and $\si\in S_d$ the image of 
$\tr(x_{\si(1)}^{\de_1}\ldots x_{\si(d)}^{\de_d})$ is $1$ in case $\de_1=\ldots=\de_d$ and the image is zero otherwise. 

\begin{lemma}\label{lemma2}
Assume that $0<p\leq \frac{n-1}{2}$. Then $\mu(f)=0$ for all $f\in I_{n,d}$.
\end{lemma}
\begin{proof}
It is enough to verify the statement of the lemma for elements of the types (A), (B) and (C) from Lemma~\ref{L2}. Clearly the lemma holds for elements of the types (A) and (B). 

Consider a multilinear triple $(\un{u};\un{v};\un{w})$ of $\L$ with $|\un{u}|=t$, $|\un{v}|=r$ and $t+2r=n+1$, and denote $f=\si_{\rm lin}(\un{u};\un{v};\un{w})$. 
Denote by $\Omega_1$ the set of all $a\in\Omega$ with $\mu(\tr(a(\un{u};\un{v};\un{w})))=1$. We have the following two possibilities. 

\medskip
\noindent{\bf(1)} If there exists an element $c$ of the set of all $u_i$, $v_j$, $w_j$ such that $c$ contains both $x_k$ and $x_q^T$ for some $k,q$, then obviously $\mu(\tr(a(\un{u};\un{v};\un{w})))=0$ for any multilinear path $a$ in $\Q$; therefore, $\mu(f)=0$. 

\medskip
\noindent{\bf(2)} For every $i,j$ there are $\hat{u}_i, \hat{v}_j, \hat{w}_j$  from $\Y$ such that 
\begin{enumerate}
\item[$\bullet$] $u_i=\hat{u}_i^{\de_i}$, $v_i=\hat{v}_j^{\theta_j}$, $w_j=\hat{w}_j^{\zeta_j}$ for some $\de_i$, $\theta_j$, $\zeta_j$ from  $\{1,T\}$; 

\item[$\bullet$] $\hat{u}_i, \hat{v}_j, \hat{w}_j$ are products of $x_1,\ldots,x_d$.
\end{enumerate}

We start with the partial case of $\de_i=\theta_j=\zeta_j=1$ for all $i,j$ and then consider the general case.

\medskip
\noindent{\bf(2.1)} Consider the partial case when $u_i$, $v_j$, $w_j$ are products of $x_1,\ldots,x_d$ for all $i,j$.

Let $r=0$. Then $t>n$ and $\Omega$ consists of all monomials $\al_1 \al_{\si(2)} \cdots \al_{\si(t)}$ for permutations $\si\in S_t$ with $\si(1)=1$. Hence $\mu(f)= (-1)^t (t-1)! =0$.

Let $r>0$. Then a path $a$ in $\Q$ belongs to $\Omega_1$ if and only if
$$a=\be_1 \ga_{\tau(1)} a_1\; \be_{\si(2)} \ga_{\tau(2)} a_2 \cdots \be_{\si(r)} \ga_{\tau(r)} a_r,$$
where $\si,\tau\in S_r$ with $\si(1)=1$ and $a_1,\ldots,a_r$ are products (possibly empty) of $\al_1,\ldots,\al_t$ such that $a_1\cdots a_r=\al_{\eta(1)}\cdots \al_{\eta(t)}$ for some permutation $\eta\in S_t$.
Thus 
$$|\Omega_1| = (r-1)! r! \cdot t!\; b_{t,r},$$
where $b_{t,r}$ stands for the number of vectors of non-negative integers $(i_1,\ldots,i_r)$ with $i_1+\cdots+i_r=t$. Since 
$b_{t,r}=\left(\begin{array}{c}
t+r-1\\
r-1\\
\end{array}
\right)$, then  
$$\mu(f)=(-1)^t\; |\Omega_1|=(-1)^t (t+r-1)!\, r!.$$
Note that $\frac{t}{2}+r=\frac{n+1}{2}> p$. If $t=0$ or $t=1$, then $r\geq p$ and $\mu(f)=0$. On the other hand, if $t>1$, then $t-1\geq \frac{t}{2}$; therefore, $\mu(f)=0$.

\medskip
\noindent{\bf(2.2)} Now we consider the general case. By Lemma~\ref{Lnew1}, without loss of generality, we assume that 
\begin{enumerate}
\item[$\bullet$]  $\de_1=\cdots=\de_{k}=1$ and $\de_{k+1}=\cdots=\de_t=T$ for some $0\leq k\leq t$;

\item[$\bullet$]  $\theta_1=\cdots=\theta_{l}=1$ and $\theta_{l+1}=\cdots=\theta_r=T$ for some $0\leq l\leq r$; 

\item[$\bullet$]  $\zeta_1=\cdots=\zeta_{m}=1$ and $\zeta_{m+1}=\cdots=\zeta_r=T$ for some $0\leq m\leq r$,
\end{enumerate}
where by Lemma~\ref{L3} we can assume that $k>0$ in case $t>0$ and $l>0$ in case $r>0$. Moreover, we assume that $k<t$ or $l<r$ or $m<r$, since the case of $k=t$ and $l=m=r$ has already been considered in part (2.1) of the proof. 

Let $r=0$. Then $t=n+1$, $0<k<t$, and $\Omega_1$ is empty; therefore, $\mu(f)=0$.

Let $r>0$. For every $1\leq j\leq r$ denote $b_j=\be_j^{\theta_j}$ and  $c_j=\ga_j^{\zeta_j}$. Since $b_1=\be_1$, then a path $a$ in $\Q$ belongs to $\Omega_1$ if and only if
$$a=\be_1 x_1 c_{\tau(1)} y_1\cdot b_{\si(2)} x_2 c_{\tau(2)} y_2 \cdots b_{\si(r)} x_r c_{\tau(r)} y_r,$$
where $\si,\tau\in S_r$ with $\si(1)=1$ and 
\begin{enumerate}
\item[$\bullet$] $x_1,\ldots,x_r$ are products (possibly empty) of $\al_{k+1}^T,\ldots,\al_{t}^T$ such that $x_1\cdots x_r=\al_{\eta(k+1)}^T\cdots \al_{\eta(t)}^T$ for some permutation $\eta$ of the set $\{k+1,\ldots,t\}$;

\item[$\bullet$]  $y_1,\ldots,y_r$ are products (possibly empty) of $\al_1,\ldots,\al_{k}$ such that $y_1\cdots y_r=\al_{\eta(1)}\cdots \al_{\eta(k)}$ for some permutation $\eta\in S_{k}$.
\end{enumerate}
Thus
$$|\Omega_1| = (r-1)! r! \cdot (t-k)!\; b_{t-k,r}\cdot k!\; b_{k,r} = 
r\; (t-k+r-1)!\; (k+r-1)!.$$
By the definition of $\xi_a$,
$$\mu(f)=(-1)^{t+l+m}\,|\Omega_1|.$$
Note that the condition $\frac{t}{2}\geq k$ implies $t-k+r-1\geq \frac{t}{2} + r -1 =\frac{n}{2}-\frac{1}{2}\geq p$ and $\mu(f)=0$. On the other hand, the condition $k\geq\frac{t}{2}$ implies $k+r-1\geq \frac{t}{2} + r -1 \geq p$ and $\mu(f)=0$. Thus, $\mu(f)=0$.
\end{proof}

\begin{proof} Now we can prove part (b) of Theorem~\ref{theo1}. Assume that $\tr(X_1^{-}\cdots X_d^{-})$ is decomposable in $R_{-}^{O(n)}$. Then $F=\tr((X_1-X_1^T)\cdots (X_d-X_d^T))$ is decomposable in $R^{O(n)}$. Note that
$$F=\sum (-1)^{\rho_{\un{\de}}}\tr(X_1^{\de_1}\cdots X_d^{\de_d}),$$
where the sum ranges over all $\un{\de}=(\de_1,\ldots,\de_d)$ with  $\de_1,\ldots,\de_d\in\{1,T\}$ and ${\rho_{\un{\de}}}$ stands for the number of symbols ``T''{} in $\un{\de}$. Hence 
$$f=\sum (-1)^{\rho_{\un{\de}}}\tr(x_1^{\de_1}\cdots x_d^{\de_d})\in\tr\L$$
belongs to $I_{n,d}$, where the sum ranges over all $\de_1,\ldots,\de_d$ from $\{1,T\}$. Since $\mu(f)=1+(-1)^d\neq 0$ for even $d$, we obtain a contradiction to Lemma~\ref{lemma2}.
\end{proof}

\end{document}